\documentclass[12pt]{amsart}

\usepackage[left=15mm, right=15mm, top=15mm, bottom=15mm]{geometry}

\usepackage[english]{babel}
\usepackage[T1]{fontenc}
\usepackage[utf8]{inputenc}

\theoremstyle{plain}
    \newtheorem{theorem}{Theorem}[section]
    \newtheorem{lemma}[theorem]{Lemma}

    \newtheorem{quest}[theorem]{Question}

\theoremstyle{definition}

    \newtheorem{example}[theorem]{Example}

\usepackage[
    draft = false,
    unicode = true,
    colorlinks = true,
    allcolors = blue,
    hyperfootnotes = true,
    citecolor = red
]{hyperref}
\usepackage{amsfonts,amsmath,amssymb,amscd,amsthm,url}
\usepackage[inline]{enumitem}
\usepackage{graphicx,epsf,afterpage, wrapfig}
\usepackage{soul}
\usepackage{multicol}
\usepackage{array}
\usepackage{braket}
\usepackage{epigraph}
\usepackage{rotating, floatflt}
\usepackage{xstring} 

\usepackage[
backend=biber,
style=numeric-comp,
sorting=none,
giveninits=true
]{biblatex}
\usepackage{csquotes}

\usepackage{xcolor}

\usepackage{ulem}

\usepackage{tikz}
\usetikzlibrary{positioning, calc, intersections, through, arrows, matrix, chains, math}
\usetikzlibrary{arrows.meta}
\usetikzlibrary{shadows}
\usetikzlibrary{graphs}
\usetikzlibrary{graphs.standard}
\usetikzlibrary{patterns}
\usetikzlibrary{decorations.pathreplacing,calligraphy,backgrounds}

\DeclareFieldFormat{usera}{\href{https://arxiv.org/abs/#1}{\texttt{arXiv:#1}}}

\renewbibmacro*{finentry}{\printfield{usera}\newunit\finentry}

\renewbibmacro{in:}{%
  \iffieldundef{usera}
    {\printtext{\bibstring{in}\intitlepunct}}
    {}%
}

\newcommand\norm[1]{\ensuremath{\left\lVert#1\right\rVert}}
\newcommand\abs[1]{\ensuremath{\left\lvert#1\right\rvert}}

\DeclareMathOperator{\vspan}{span}

\newcommand{\Acal}{\mathcal{A}}
\newcommand{\Bcal}{\mathcal{B}}

\newcommand{\Hcal}{\mathcal{H}}

\newcommand{\Mcal}{\mathcal{M}}

\newcommand{\Ucal}{\mathcal{U}}

\DeclareMathOperator{\Idbf}{\mathbf{Id}}

\newcommand{\Abf}{{\ensuremath{\mathbf{A}}}}
\newcommand{\Bbf}{{\ensuremath{\mathbf{B}}}}

\newcommand{\Ebf}{{\ensuremath{\mathbf{E}}}}

\newcommand{\Ibf}{{\ensuremath{\mathbf{I}}}}

\newcommand{\Pbf}{{\ensuremath{\mathbf{P}}}}
\newcommand{\Qbf}{{\ensuremath{\mathbf{Q}}}}
\newcommand{\Rbf}{{\ensuremath{\mathbf{R}}}}

\newcommand{\Ubf}{{\ensuremath{\mathbf{U}}}}

\newcommand{\Phbf}{{\ensuremath{\mathbf{\Phi}}}}

\newcommand{\defing}[1]{\textbf{\emph{\mathversion{bold}#1}}}

\newcommand{\R}{\ensuremath{\mathbb{R}}}

\newcommand{\Cx}{\ensuremath{\mathbb{C}}}

\newcommand{\N}{\ensuremath{\mathbb{N}}}

\renewcommand{\geq}{\geqslant}

\renewcommand{\leq}{\leqslant}

\newcounter{mcnt}

\newcounter{wordcnt}

\begin{document}

\title[{Quantum channels on duals of von Neumann algebras in the Schr\"{o}dinger picture}]
{Quantum channels on duals of von Neumann algebras \\ in the Schr\"{o}dinger picture}

\author{Sviatoslav V. Dzhenzher}

\begin{abstract}
    The theory of quantum channels is traditionally studied either on finite-dimensional state spaces or within the Heisenberg picture as completely positive maps on $C^*$-algebras. In this paper, we consider quantum channels as completely positive maps on the duals of general von Neumann algebras in the Schr\"{o}dinger picture.
    We investigate the construction of such channels through Pettis integrals using representations of topological groups.
\end{abstract}

\thanks{\hspace{-4mm}
S.\,V. Dzhenzher: sdjenjer@yandex.ru. orcid: 0009-0008-3513-4312
\\
Moscow Institute of Physics and Technology 141701, Institutskii lane, 9, Dolgoprudny, Russia}

\maketitle
\thispagestyle{empty}


\vspace{3mm}

\section{Introduction}

In quantum physics, the fundamental task is to investigate the time evolution of physical systems \cite{Dirac-princ-qm}.
While an isolated system can be elegantly described by a unitary evolution via the Schr\"{o}dinger equation, realistic quantum systems are never perfectly shielded from their surrounding environments.
When we consider such non-isolated, open quantum systems \cite{Accardi-Volovich-Lu-2002-open-sys}, their dynamics can no longer be captured by purely unitary operators, which naturally necessitates the formalism of quantum channels to model state transitions, decoherence, and noise.
Thus, the theory of quantum channels serves as a cornerstone for quantum information processing and open quantum systems \cite{Kholevo-qsys-chan-inf, Nielsen-Chuang}.
However, the vast majority of literature in theoretical physics focuses exclusively on quantum channels acting on finite-dimensional state spaces.
When infinitely-dimensional spaces arise, mostly positive quantum states are considered on them \cite{Akemann-Anderson-Pedersen-1986, Bratteli-Robinson}.
While this framework is sufficient for many practical applications, it fails to capture the full algebraic complexity of systems with infinite degrees of freedom. 

Channels on general von Neumann algebras are studied significantly less. When infinite-dimensional systems are considered, the conventional approach typically relies on the Heisenberg picture, where quantum channels are modelled as completely positive unital maps on $C^*$-algebras \cite{Kholevo-qprobstat}. This leaves a notable gap in the systematic development of the dual description. 

To address this, some recent works have begun shifting focus towards the Schr\"{o}dinger picture in a more general way.
For instance, in \cite{Amosov-Sakbaev-2013, Amosov-Sakbaev-2025}, quantum channels on whole state spaces (including non-normal states) are investigated; and in \cite{Blecher-Weaver}, the detailed analysis of state spaces on non-separable Hilbert spaces is provided.
Following this direction, in this paper, we consider quantum channels rigorously defined as completely positive maps on the duals of von Neumann algebras, specifically in the Schr\"{o}dinger picture.
As we will demonstrate, specific instances of these channels naturally arise within the framework of singular quantum states developed in \cite{DzhenzherDzhenzherSakbaev26, Dzhenzher26-qsing}.

The paper is organised into five sections.
Section~\ref{s:defs} reviews the core definitions and standard examples of quantum channels.
Section~\ref{s:q-ch} establishes our main results (theorems~\ref{t:q-ch} and~\ref{t:norm-top}), discusses physically motivated examples, and analyses a key non-existence case in Example~\ref{ex:non-meas}.
Section~\ref{s:semigroups} examines the corresponding quantum channel semigroups.
Section~\ref{s:singular} discusses the singularising properties of the constructed quantum channels.
Finally, Section~\ref{s:concl} outlines prospective research paths.

\section{Quantum channels on von Neumann algebras}\label{s:defs}

Let $\Hcal$ be a Hilbert space with the inner product \((\cdot,\cdot)\) being linear in the second argument.

Denote by \(\Bcal(\Hcal)\) the Banach algebra of linear bounded operators \(\Hcal\to\Hcal\).
Let \(\Mcal\subset\Bcal(\Hcal)\) be a von Neumann algebra of \defing{observables}; that is, a weakly closed $*$-algebra of $\Bcal(\Hcal)$ containing the \emph{identity operator} $\Ibf$.
Throughout the paper, \(\Mcal^*\) denotes the Banach dual of $\Mcal$ (but not the predual).

Recall that the observable \(\Abf\in\Mcal\) is \emph{positive} if \((x,\Abf x)\geq 0\) for any \(x\in\Hcal\),
and the functional \(\rho\in\Mcal^*\) is \emph{positive} if \(\Braket{\rho, \Abf} \geq 0\) for any positive \(\Abf\in\Mcal\);
we will denote these \(\Abf\geq 0\) and \(\rho\geq 0\), respectively.
Denote by \(\Sigma(\Mcal) := S^+_1(\Mcal^*)\) the space of \defing{quantum states}, that is, the linear continuous functionals \(\Mcal \to \Cx\), lying at the intersection of the positive cone and the unit sphere.
In other words,
\[
    \Sigma(\Mcal) := \left\{\rho\in\Mcal^*\::\:\rho\geq0,\,\Braket{\rho,\Ibf}=1\right\}.
\]
We denote the \defing{action} of a functional \(\rho \in \Mcal^*\) on an observable \(\Abf\in \Mcal\) by
\[
    \braket{\rho, \Abf} := \rho(\Abf).
\]

For any integer \(n\geq 1\), denote by \(\Mcal_n\) the algebra of complex matrices \(n\times n\).
Note that for a von Neumann algebra \(\Mcal \subset \Bcal(\Hcal)\), the tensor product \(\Mcal\otimes\Mcal_n\) is also a von Neumann algebra of matrices \(n\times n\), whose elements are from \(\Mcal\).
That is, an operator \(\Abf\in\Mcal\otimes\Mcal_n\) acts in a Hilbert space \(\Hcal^n = \bigoplus_{k=1}^n\Hcal\) in a way similar to the usual action of finite matrices on \(\R^n\).
Also, identifying \(\Mcal_n^*\cong\Mcal_n\), note that \((\Mcal\otimes\Mcal_n)^*\cong \Mcal^*\otimes\Mcal_n^* \cong \Mcal^* \otimes \Mcal_n\); see, for example, \cite[Chapter~IV]{Takesaki1979}.

Let \(\Phbf\colon \Mcal^*\to\Mcal^*\) be a linear map.
Recall that \(\Phbf\) is \emph{positive} if \(\Phbf\rho \geq 0\) for any \(\rho\geq 0\).
Following the Heisenberg picture described in \cite{Kholevo-qprobstat}, and both Heisenberg and Schr\"{o}dinger pictures described in \cite{Takesaki1979}, we say that \(\Phbf\) is \defing{completely positive} if for any integer \(n\geq 1\), the map \(\Phbf\otimes\Idbf_n\colon (\Mcal\otimes\Mcal_n)^* \to (\Mcal\otimes\Mcal_n)^*\) defined by
\[
    \Phbf\otimes\Idbf_n(\rho\otimes \sigma) := \Phbf(\rho) \otimes \sigma
\]
is positive.
A completely positive linear map \(\Phbf\colon\Mcal^*\to\Mcal^*\) is called a \defing{quantum channel} if the state space \(\Sigma(\Mcal)\) is invariant under \(\Phbf\), which means \(\Phbf(\Sigma(\Mcal))\subset \Sigma(\Mcal)\).
Note that it differs from the usual definition of quantum channels, where the space of so called normal states should be invariant.
This invariance is the analogue of the trace-preserving property, which is classical when speaking of finite dimensional Hilbert spaces.

In theoretical physics, it is a well known fact that a unitary evolution is a quantum channel.
In our case, in order to have this property, we will need the following notion: we say that a unitary evolution \(\Ubf\in\Ucal(\Hcal)\) \defing{remains in the von Neumann algebra} \(\Mcal\subset\Bcal(\Hcal)\) if \(\Ubf^*\Mcal\Ubf \subset \Mcal\).
For example, if \(\Mcal=\Bcal(\Hcal)\), then all unitary evolutions from \(\Ucal(\Hcal)\) remain in \(\Mcal\); see other examples of unitary evolutions remaining in von Neumann algebras in Section~\ref{s:q-ch}.

\begin{lemma}\label{l:unit-q-ch}
    Let \(\Ubf\in\Ucal(\Hcal)\) be a unitary evolution that remains in a von Neumann algebra \(\Mcal\subset\Bcal(\Hcal)\).
    Then the map \(\Phbf_\Ubf\colon\Mcal^*\to\Mcal^*\) given by
    \[
        \Phbf_\Ubf\rho := \Ubf\rho\Ubf^*
    \]
    is a quantum channel.
\end{lemma}

\begin{proof}
    First, let us prove the correctness of the definition.
    Since the unitary evolution given by \(\Ubf\) remains in \(\Mcal\),
    for any \(\rho\in\Mcal^*\), the unitary evolution \(\Ubf\rho\Ubf^*\), which is given by
    \[
        \Braket{\Ubf\rho\Ubf^*,\Abf} = \Braket{\rho, \Ubf^*\Abf\Ubf},
    \]
    remains in \(\Mcal^*\).
    Hence \(\Phbf_\Ubf\) is well defined.

    Now, the linearity and positivity are clear.
    It is also clear that \(\Phbf_\Ubf\Sigma(\Mcal)\subset \Sigma(\Mcal)\), since
    \[
        \Braket{\Phbf_\Ubf\rho, \Ibf} = \Braket{\rho, \Ubf^*\Ibf\Ubf} = \Braket{\rho, \Ibf}=1.
    \]
    Finally, the complete positivity follows since for any integer \(n\geq 1\) the map
    \[
        \Phbf_\Ubf\otimes \Idbf_n = \Phbf_{\Ubf\otimes \Ibf_{n\times n}}
    \]
    is positive.
\end{proof}

Below, we give some examples of generalisations of classical quantum channels on finite-dimensional Hilbert spaces to arbitrary Hilbert spaces.

\begin{example}
    Any finite convex combination of quantum channels is a quantum channel.
\end{example}

\begin{example}[Depolarising channel]
    Let \(\Hcal\) be a Hilbert space, and \(\Mcal\subset\Bcal(\Hcal)\) be a von Neumann algebra.
    If $\Hcal$ is infinite-dimensional, there exists no ``maximally mixed state'' \(\frac{\Ibf}{dim}\).
    So, instead of this, fix any state \(\rho_0\in\Sigma(\Mcal)\), \(q\in(0,1)\), and define the channel by
    \[
        \Phbf\rho := (1-q)\rho + q\Braket{\rho, \Ibf}\rho_0.
    \]
\end{example}

\begin{example}[Erasure channel]
    Let \(\Hcal\) be a Hilbert space, and \(\Mcal\subset\Bcal(\Hcal)\) be a von Neumann algebra.
    Let \(e=\ket{e}\in\Cx^2\) be a unit vector.
    In the space \(\Hcal\oplus \Cx e\), consider the von Neumann algebra \(\widetilde\Mcal := \Mcal \oplus \Cx\).
    For a functional \(\rho\in \Mcal^*\), define \(\widetilde\rho\in(\widetilde\Mcal)^*\) by
    \[
        \Braket{\widetilde\rho,\Abf \oplus c} := \Braket{\rho,\Abf}.
    \]
    Also, define the fixed functional \(\widetilde\rho_e\in(\widetilde\Mcal)^*\) by
    \[
        \Braket{\widetilde\rho_e,\Abf \oplus c} := c.
    \]
    Then, the erasure channel for \(q\in(0,1)\) is
    \[
        \Phbf \rho := (1-q)\widetilde\rho + q \Braket{\rho,\Ibf} \widetilde\rho_e.
    \]
    Strictly speaking, as in the classical physics, this is the quantum channel \(\Mcal^*\to(\widetilde\Mcal)^*\).
\end{example}

\begin{example}[Dephasing channel]
    Let \(\Hcal\) be a Hilbert space, and \(\Mcal\subset\Bcal(\Hcal)\) be a von Neumann algebra.
    Let \(\Acal\subset\Mcal\) be its MASA, and \(\Ebf \colon \Mcal\to\Acal\subset\Mcal\) be a normal conditional expectation; see necessary definitions in \cite{Farah-Wofsey, Takesaki2003}.
    Then \(\Ebf^*\colon\Mcal^*\to\Mcal^*\) is a quantum channel, and for \(q\in(0,1)\) so is
    \[
        \Phbf := (1-q)\Idbf + q\Ebf^*.
    \]
\end{example}



\section{Quantum channels as Pettis integrals}\label{s:q-ch}

In the theorems below, we give the core examples of quantum channels.
They will be constructed as averages of unitary evolutions by some measure on a group $G$ and a representation \(\pi\colon G \to\Ucal(\Hcal)\),
where \(\Ucal(\Hcal)\subset\Bcal(\Hcal)\) is the group of unitary operators on \(\Hcal\).
Recall that the Pettis integral
\[
    \rho := \int \rho_x\,d\mu(x)
\]
is defined by the actions
\[
    \Braket{\rho,\Abf} = \int \Braket{\rho_x,\Abf}\,d\mu(x).
\]
For the definition and properties of integrals over finitely additive measures, see, for example, \cite{Dunford-Schwartz-vol1}.

\begin{theorem}\label{t:q-ch}
    Let \(\Hcal\) be a Hilbert space.
    Let \(G\) be a group.
    Let \(\pi\colon G\to\Ucal(\Hcal)\) be a representation of groups such that all unitary evolutions from \(\pi(G)\) remain in the von Neumann algebra \(\Mcal\subset\Bcal(\Hcal)\).
    
    Let \(\mu\colon2^G\to[\,0,1\,]\) be a finitely additive measure defined on all subsets of \(G\) with \(\mu(G)=1\).
    Define the map \(\Phbf_{\pi,\mu}\colon \Mcal^*\to\Mcal^*\) by the Pettis integral
    \[
        \Phbf_{\pi,\mu} \rho := \int \pi(g)\rho\pi(g)^*\,d\mu(g).
    \]
    
    Then the map \(\Phbf_{\pi,\mu}\) is a quantum channel.
\end{theorem}

\begin{proof}
    The linearity, complete positivity, and invariance of \(\Sigma(\Mcal)\) follow by Lemma~\ref{l:unit-q-ch}, since the Pettis integral respects all these properties.
\end{proof}

Theorem~\ref{t:q-ch} implies that the linear maps discussed in \cite{Dzhenzher26-qsing} are quantum channels in the sense of this article.
In particular, this covers the case when \(\pi(g)\) are the shift operators.

Now we will focus our attention on the measures that are not defined on the whole power set \(2^G\).
One may find this more natural since it covers the cases of the Lebesgue measure or probability measures on~\(\R\).

\begin{theorem}\label{t:norm-top}
    Let \(\Hcal\) be a Hilbert space.
    Let \(G\) be a topological group, and \(\Bcal_G \subset 2^G\) be the Borel $\sigma$-algebra of the open subsets of~$G$.
    Let \(\pi\colon G\to\Ucal(\Hcal)\) be a representation of groups, which is continuous in the norm topology,
    and such that all unitary evolutions from \(\pi(G)\) remain in the von Neumann algebra \(\Mcal\subset\Bcal(\Hcal)\).
    
    Let \(\mu\colon\Bcal_G\to[\,0,1\,]\) be a finitely additive measure with \(\mu(G)=1\).
    Define the map \(\Qbf_{\pi,\mu}\colon \Mcal^*\to\Mcal^*\) by the Pettis integral
    \[
        \Qbf_{\pi,\mu} \rho := \int \pi(g)\rho\pi(g)^*\,d\mu(g).
    \]
    
    Then the map \(\Qbf_{\pi,\mu}\) is a quantum channel.
\end{theorem}

\begin{proof}
    The continuity of \(\pi\) gives that for any \(\Abf\in\Mcal\), the mapping
    \[
        g\mapsto \Braket{\pi(g)\rho\pi(g)^*, \Abf}
    \]
    is continuous, and hence is \(\Bcal_G\)-measurable.
    Now the remaining part of the proof is as of the proof of Theorem~\ref{t:q-ch}.
\end{proof}

Cf.~Theorem~\ref{t:norm-top} to the results on quantum channels in \cite{DzhenzherDzhenzherSakbaev26}.

\begin{example}
    Let \(\Hcal = \Cx^2\) be the $1$-qubit space.
    Let \(G=SU(2)\) be the group of rotations, and \(\pi\colon g\mapsto g\).
    Let \(\Mcal=\Bcal(\Hcal)\).
    Then the assumptions of Theorem~\ref{t:norm-top} are satisfied for any Borel measure $\mu$.
\end{example}

\begin{example}
    Let \(\Hcal = \ell_2(\N)\), and \(\Mcal=\Bcal(\Hcal)\).
    Let \(G = (\R,+)\), and \(\pi(g)\) be the operator of multiplication on the function \(n\mapsto e^{\frac{ig}{n}}\).
    Then \(\pi\) is continuous in the norm topology.
    Indeed,
    \[
        \norm{\pi(g)-\Ibf} = \sup_{n\in\N} \abs{e^{\frac{ig}{n}} - 1} \leq \abs{g} \to 0.
    \]
\end{example}

\begin{example}
    Let \(\Hcal = \ell_2(\R)\) be a space of functions \(f\colon\R\to\Cx\) with the at most countable support, and such that the sequence of values \(f(x)\) for $x$ from the support of $f$ lies in the space \(\ell_2(\N)\) of sequences.
    Let \(G = (\R,+)\), and \(\pi(g)\) be the operator of multiplication on the function \(x\mapsto e^{ig\cos(x)}\).
    Then \(\pi\) is continuous in the norm topology.
    Indeed,
    \[
        \norm{\pi(g)-\Ibf} = \sup_{z\in[\,-1,1\,]} \abs{e^{igz} - 1} \leq \sup_{z\in[\,-1,1\,]} \abs{gz} \xrightarrow[g\to 0]{}0.
    \]

    Let \(\Rbf \in \Bcal(\Hcal)\) be the reflection operator acting like
    \[
        \Rbf f(x) = f(-x).
    \]
    Let \(\Mcal\) be a von Neumann algebra of operators commuting with \(\Rbf\).
    Then all unitary evolutions from \(\pi(G)\) remain in \(\Mcal\), since \(\pi(g)\) commutes with \(\Rbf\).

    As another example, one may take \(\Mcal\) as the direct sum \(\bigoplus_{k=1}^n\Bcal(\ell_2(X_k))\) for some decomposition \(\R=\bigsqcup_{k=1}^n X_k\).
\end{example}

\begin{example}
    Analogously to the previous example, one may take \(\Hcal = L_2([\,-1,1\,])\).
\end{example}

It is natural to ask whether in Theorem~\ref{t:norm-top} the SOT-continuous representation $\pi$ can be taken.
Unfortunately, this is not true even in case of separable Hilbert spaces, as the example below shows.

\begin{example}\label{ex:non-meas}
    Let \(\Hcal:= \ell_2(\N) \otimes \Cx^2\).
    Let \(\{e_n\}\) be a standard basis in \(\ell_2(\N)\).
    Let \(G\) be an infinite torus \((\mathbb{T}^\N, \cdot)\), which consists of elements \(g=(e^{i\theta_1},e^{i\theta_2},\ldots)\), with the product topology.
    Let
    \[
        \Ubf(\theta) :=
        \begin{pmatrix}
            \cos\theta & -\sin\theta \\ \sin\theta & \cos\theta
        \end{pmatrix}
    \]
    be a rotation matrix in \(\Cx^2\).
    Define the representation \(\pi\colon G\to\Ucal(\Hcal)\) by
    \[
        \pi(g)(e_n\otimes v) := e_n \otimes \Ubf(\theta_n)v.
    \]
    Below we show that \(\pi\) is SOT-continuous.
    Indeed, take any \(g_\alpha\to g\).
    For \(x = e_n\otimes v\) it is clear that \(\pi(g_\alpha)x \to \pi(g)x\).
    Then it is clear for any \(x\in \Hcal_0:= \vspan(e_n\otimes v)\).
    Finally, for arbitrary \(x\in\Hcal\), take \(x_0\in\Hcal_0\) which is \(\frac{\varepsilon}{3}\)-close to $x$.
    Then
    \[
        \norm{\pi(g_\alpha)x - \pi(g)x} \leq
        \norm{\pi(g_\alpha)(x-x_0)} + \norm{\pi(g_\alpha)x_0 - \pi(g)x_0} + \norm{\pi(g)(x_0-x)} \leq
        \norm{\pi(g_\alpha)x_0 - \pi(g)x_0} + \frac{2\varepsilon}{3}.
    \]
    It remains to take \(\alpha\) large enough so that \(\norm{\pi(g_\alpha)x_0 - \pi(g)x_0} < \frac{\varepsilon}{3}\).

    Now let \(\Mcal := \ell_\infty(\N)\otimes \Mcal_2\) be the von Neumann algebra of block-diagonal operators.
    It is clear from the definition that \(\pi(g)\) remains in $\Mcal$.
    Take \(\Abf := I \otimes \begin{pmatrix}
        1 & 0 \\ 0 & -1
    \end{pmatrix}\).
    Then \(\pi(g)^*\Abf\pi(g)\) is the block-diagonal operator, with its $n$-th block equal to
    \[
        \Ubf(\theta_n)^*\begin{pmatrix}
        1 & 0 \\ 0 & -1
        \end{pmatrix}\Ubf(\theta_n) =
        \begin{pmatrix}
            \cos(2\theta_n) &-\sin(2\theta_n) \\ -\sin(2\theta_n) & -\cos(2\theta_n)
        \end{pmatrix}.
    \]

    Now let \(U\) be a non-principal ultrafilter on $\N$.
    Define the state \(\rho\in\Mcal^*\) by
    \[
        \Braket{\rho, \Bbf} := \lim_{n\to U} (\Bbf_n)_{1,1},
    \]
    where \((\Bbf_n)_{1,1},\) denotes the upper-left element of the $n$-th $2\times 2$-block of $\Bbf$.
    Hence,
    \[
        \Braket{\rho, \pi(g)^*\Abf\pi(g)} = \lim_{n\to U} \cos(2\theta_n).
    \]

    In the rest of the example, we show that the map
    \[
        \psi\colon g \mapsto \lim_{n\to U} \cos(2\theta_n)
    \]
    is not \(\Bcal_G\)-measurable.
    Consider the subset \(G_0\subset G\) of elements $g$ for which \(\theta_n \in \{0,\frac{\pi}{2}\}\), where \(\pi\approx 3.14\) (and not the representation).
    Then $G_0$ is homeomorphic to the classical Cantor set \(\{0,1\}^\N\); in particular, it is closed.
    Now, for any \(g\in G_0\), let \(N_g\subset \N\) be a subset of integers such that
    \[
        n \in N_g \Longleftrightarrow \theta_n = 0.
    \]
    Hence, for \(g\in G_0\)
    \[
        \psi(g) = \lim_{n\to U}\cos(2\theta_n) = \begin{cases}
            1, &N_g \in U. \\ -1, &N_g \notin U.
        \end{cases}
    \]
    Therefore,
    \[
        \psi^{-1}(1) \cap G_0 = \{g \in G_0 : N_g \in U\} \cong U,
    \]
    where in the latter equality we identify the elements of \(\{0,1\}^\N\) with the subsets of $\N$.
    It remains to use the folklore fact that any non-principal ultrafilter on $\N$ does not have the Baire property, and thus is not Borel measurable (the other argument can be that by Sierpi\'{n}ski's theorem \cite[Theorem~6.1]{Blass2010}, a non-principal ultrafilter on $\N$ is not Lebesgue measurable, and thus again is not Borel measurable). This contradicts the fact that the intersection of a Borel pre-image \(\psi^{-1}(1)\) and the closed set $G_0$ should be Borel.
\end{example}

\section{Quantum channels semigroups}\label{s:semigroups}

In this section, following the ideas from \cite{DzhenzherDzhenzherSakbaev26, Amosov-Sakbaev-2025}, we consider the quantum channels from the point of view of the semigroups.
We focus our attention on the families \(\{\mu_t\}_{t\geq 0}\) of $\sigma$-additive measures that form semigroups by convolution, which means,
\[
    \mu_s * \mu_t = \mu_{s+t}
    \quad\text{for all $s,t\geq 0$}
\]
Recall that this means that for any integrable function \(f\colon G\to\Cx\),
\[
    \int f(g)\,d\mu_{s+t}(g) = \int d\mu_t(u) \int f(uv)\,d\mu_s(v).
\]
Note that \(\mu_s * \mu_t = \mu_{s+t} = \mu_t*\mu_s\) even for non-abelian groups, since we may just swap $s$ and $t$.

\begin{example}
    Take \(G=(\R,+)\) and \(\mu_t\) to be the Gaussian measures with zero mean and variances equal to $t$.
\end{example}

\begin{theorem}
    Let \(\{\mu_t\}_{t\geq 0}\) be the family of $\sigma$-additive measures, that form semigroups by convolution.
    Then, under the assumptions of Theorem~\ref{t:q-ch}, the quantum channels \(\left\{\Phbf_{\pi,\mu_t}\right\}_{t\geq0}\) form a semigroup;
    and under the assumptions of Theorem~\ref{t:norm-top}, the quantum channels \(\left\{\Qbf_{\pi,\mu_t}\right\}_{t\geq0}\) form a semigroup.
\end{theorem}

\begin{proof}
    For any \(\rho\in\Mcal^*\) and \(\Abf\in\Mcal\), we have
    \begin{multline*}
        \Braket{\Phbf_{\pi,\mu_{s+t}}\rho, \Abf} =
        \int \Braket{\pi(g)\rho\pi(g)^*, \Abf}\,d\mu_{s+t}(g) = \\ =
        \int d\mu_t(u)\int \Braket{\pi(uv)\rho\pi(uv)^*, \Abf}\,d\mu_s(v) =
        \int d\mu_t(u)\int \Braket{\pi(v)\rho\pi(v)^*, \pi(u)^*\Abf\pi(u)}\,d\mu_s(v) = \\ =
        \int d\mu_t(u) \Braket{\Phbf_{\pi,\mu_s}\rho, \pi(u)^*\Abf\pi(u)} = \Braket{\Phbf_{\pi,\mu_t}\Phbf_{\pi,\mu_s}\rho,\Abf}.
    \end{multline*}
    Hence
    \[
        \Phbf_{\pi,\mu_{s+t}} = \Phbf_{\pi,\mu_t}\Phbf_{\pi,\mu_s} = \Phbf_{\pi,\mu_s}\Phbf_{\pi,\mu_t}.
    \]
    The proof for $\Qbf_{\pi,\mu_t}$ is analogous.
\end{proof}

\section{Singularising quantum channels}\label{s:singular}

With the framework of quantum channels on the entire \(\Mcal^*\), it is now possible to extend some results of \cite{Dzhenzher26-qsing} to the quantum channels discussed here.

A functional $\rho\in \Mcal^*$ is said to be \defing{normal} if it is weakly-$*$ continuous.
It is well known \cite{Takesaki1979} that a space of normal functionals is the predual \(\Mcal_*\), and that there exists the projector \(\Pbf_n\colon\Mcal^*\to\Mcal_*\subset\Mcal^*\).

Denote \(\Pbf_s:=\Ibf - \Pbf_n\).
A functional \(\rho\in\Mcal^*\) is said to be \defing{singular} if \(\rho\in\Pbf_s\Mcal^*\).

The well known result \cite[Chapter~III, Theorem~2.14]{Takesaki1979} states that any \(\rho\in\Mcal^*\) can be decomposed in a unique way in the direct sum
\[
    \rho = \rho_n + \rho_s
\]
for a positive functional \(\rho_n \in \Mcal_*\) and singular \(\rho_s \in \Pbf_s\Mcal^*\).
It is clear that \(\rho_n = \Pbf_n\rho\) and \(\rho_s=\Pbf_s\rho\).

We say that a quantum channel \(\Phbf\colon\Mcal^*\to\Mcal^*\) is \defing{singularising} if it maps any initial functional to a singular one.
The following theorem closes the question from \cite{Dzhenzher26-qsing} regarding singularisation properties of some specific channels.

\begin{theorem}
    Let \(\Hcal=\ell_2(\kappa)\) for an uncountable cardinal \(\kappa\), \(G=(\kappa,+)\), \(\pi\) be the representation for the unitary group of shift operators, and \(\mu\colon 2^\kappa\to[\,0,1\,]\) be a $\sigma$-additive probability measure that vanishes on singletons.
    Then \(\Phbf_{\pi,\mu} \colon\Bcal(\Hcal)^*\to\Bcal(\Hcal)^*\) is a singularising quantum channel.
\end{theorem}

\begin{proof}
    Take any initial functional \(\rho\in\Bcal(\Hcal)^*\).
    By \cite[remark after Definition~2.15, Chapter~III]{Takesaki1979}, there is the representation
    \[
        \rho = (\rho_1-\rho_2) + i(\rho_3-\rho_4),
    \]
    where \(\rho_i\in\Bcal(\Hcal)^*\) are positive.
    As it was shown in \cite{Dzhenzher26-qsing}, all \(\Phbf_{\pi,\mu}\rho_i\) are singular.
    Hence, \(\Phbf_{\pi,\mu}\rho\) is singular.
\end{proof}

\section{Conclusion}\label{s:concl}

In this paper, we have shifted the traditional focus of quantum channel theory from finite-dimensional spaces and the Heisenberg picture to a rigorous exploration of the Schr\"{o}dinger picture on general von Neumann algebras.
By allowing the entire state space $\Sigma(\Mcal)$, including non-normal states, to serve as an invariant subspace, this framework successfully accommodates complex infinite-dimensional algebraic structures.

We established that valid quantum channels can be systematically constructed as averages of unitary evolutions via Pettis integrals. Furthermore, we demonstrated that when the underlying measures form a semigroup under convolution, the resulting families of channels naturally inherit these semigroup properties. 

Future research may focus on a deeper characterisation of the geometric properties of the invariant state spaces under these channels.
Another promising direction is the investigation of the specific class of singular states and their asymptotic behaviour within the constructed semigroups, which could provide deeper insights into the dynamics of continuous-variable open quantum systems.

We end the paper with the following final question.
The positive answer would mean that in some cases the usage of von Neumann algebras is crucial.

\begin{quest}
    Does there exist a Hilbert space \(\Hcal\), a von Neumann algebra \(\Mcal\subset\Bcal(\Hcal)\), and a quantum channel \(\Phbf\colon\Mcal^*\to\Mcal^*\) such that there exists no ``extension'' \(\widetilde\Phbf\colon\Bcal(\Hcal)^*\to\Bcal(\Hcal)^*\) in the sense
    \[
        r\widetilde\Phbf = \Phbf r?
    \]
    Here \(r\colon \Bcal(\Hcal)^*\to\Mcal^*\) is a standard restriction \(r(\rho) = \rho|_\Mcal\).
\end{quest}

\printbibliography

\end{document}